\documentclass[12pt, reqno]{amsart}
\usepackage{fullpage}
\usepackage{amsmath, amsthm, amssymb}
\usepackage{mathrsfs}
\usepackage{enumerate}
\usepackage{todonotes}
\usepackage{tikz}
\tikzstyle{v}=[circle, draw, fill=white,
                        inner sep=0pt, minimum width=4pt]
\tikzset{>=stealth}

\newtheorem{thm}{Theorem}[section]

\newtheorem{prop}[thm]{Proposition}
\newtheorem{cor}[thm]{Corollary}

\theoremstyle{definition}

\newtheorem*{defn}{Definition}
\newtheorem{ex}[thm]{Example}

\theoremstyle{remark}
\newtheorem{rmk}[thm]{Remark}

\newcommand{\FK}{\mathcal E}

\newcommand{\CC}{\mathbf C}
\newcommand{\del}{\partial}
\newcommand{\ot}{\otimes}
\newcommand{\id}{\mathrm{id}}
\newcommand{\Sbar}{\overline{S}}
\newcommand{\taubar}{\overline{\tau}}
\renewcommand{\SS}{\mathfrak S}
\renewcommand{\subset}{\subseteq}

\begin{document}
\title{Positive expressions for skew divided difference operators}
\author{Ricky Ini Liu}
\address{Department of Mathematics, North Carolina State University, Raleigh, NC }
\email{riliu@ncsu.edu}

\begin{abstract}
	For permutations $v,w \in \SS_n$, Macdonald defines the skew divided difference operators $\del_{w/v}$ as the unique linear operators satisfying $\del_w(PQ) = \sum_v v(\del_{w/v}P) \cdot \del_vQ$ for all polynomials $P$ and $Q$. We prove that $\del_{w/v}$ has a positive expression in terms of divided difference operators $\del_{ij}$ for $i<j$. In fact, we prove that the analogous result holds in the Fomin-Kirillov algebra $\mathcal E_n$, which settles a conjecture of Kirillov.
\end{abstract}

\maketitle

\section{Introduction}

The divided difference operators $\del_{ij}$ acting on $\CC[x_1, \dots, x_n]$ are vital in the study of Schubert calculus. Their main purpose is to define Schubert polynomials, which serve as polynomial representatives of Schubert classes in the cohomology ring of the flag variety. Finding a combinatorial formula for the structure constants $c_{uv}^w$ of Schubert polynomials is a long outstanding problem in algebraic combinatorics.

Macdonald \cite{Macdonald} defined for any permutations $v$ and $w$ a skew divided difference operator $\del_{w/v}$ such that applying $\del_{w/v}$ to the Schubert polynomial of a permutation $u$ with $\ell(u) + \ell(v) = \ell(w)$ gives the structure constant $c_{uv}^w$. In \cite{Kirillov}, Kirillov conjectures that $\del_{w/v}$ can be written as a polynomial in $\del_{ij}$ for $i<j$ with positive coefficients. The main result of this paper is to prove this conjecture.

In fact, Kirillov conjectures a slightly more general result. The divided difference operators give a representation of a larger algebra $\FK_n$ introduced by Fomin and Kirillov \cite{FK}. Kirillov then conjectures that the element of the Fomin-Kirillov algebra corresponding to $\del_{w/v}$ has a positive expression in terms of generators $x_{ij} \in \FK_n$ for $i<j$. This form of positivity in $\FK_n$ is notable due to the nonnegativity conjecture in \cite{FK}, which states that certain elements of $\FK_n$ (namely evaluations of Schubert polynomials at Dunkl elements) have such a positive expression. A proof of this nonnegativity conjecture together with an explicit positive expression for such elements would immediately give a combinatorial formula for the structure constants $c_{uv}^w$.

The Fomin-Kirillov algebra also has the structure of a braided Hopf algebra as noted in \cite{FP, MS}. This added structure (which does not exist in full for the quotient algebra generated by the divided difference operators) will be key in proving our main theorem.

We begin with some preliminaries about the symmetric group, divided difference operators, and the braided Hopf algebra structure of the Fomin-Kirillov algebra $\FK_n$ in Section 2. We then prove the main result in Theorem~\ref{thm-main} of Section 3, giving a positive explicit formula for $\del_{w/v}$ in Corollary~\ref{cor-explicit} and a positive recursive formula in Corollary~\ref{cor-recurrence}.

\section{Preliminaries}

In this section, we give some notation and background, and we prove some basic facts about divided difference operators and the Fomin-Kirillov algebra. For more information, see, for instance, \cite{Humphreys, Kirillov, Macdonald}.

\subsection{Symmetric group}

Let $\SS_n$ be the symmetric group on $n$ letters. We will write $s_{ij}$ for the transposition switching $i$ and $j$, and we will abbreviate the simple transposition $s_{i,i+1}$ by $s_i$.

Given an element $w \in \SS_n$, a \emph{reduced expression} for $w$ is a factorization of $w$ into simple transpositions $s_{i_1} \dotsm s_{i_\ell}$ of minimum length $\ell=\ell(w)$. Any two reduced expressions for $w$ can be obtained from one another by commuting $s_i$ and $s_j$ when $|i-j|>1$ or by applying \emph{braid moves} replacing $s_is_js_i$ with $s_js_is_j$ when $|i-j|=1$.

We say $w = u \cdot v$ is a \emph{reduced factorization} if $\ell(w) = \ell(u)+\ell(v)$. We also denote the longest element of $\SS_n$ by $w_0$, so $\ell(w_0) = \binom{n}{2}$.

If some (equivalently, any) reduced expression for $w$ contains a subsequence that is a reduced expression for $v$, we say that $v<w$ in \emph{Bruhat order}. Equivalently, $v$ is covered by $w$ in Bruhat order, written $v \lessdot w$, if $\ell(v) = \ell(w)-1$ and $v = ws_{ij}$ for some transposition $s_{ij}$.

\subsection{Divided difference operators}

Define the left action of $\SS_n$ on $\CC[x_1, \dots, x_n]$ by \[(wP)(x_1, \dots, x_n) = P(x_{w(1)}, \dots, x_{w(n)}).\]
The \emph{divided difference operator} $\del_{ij}$ is then defined by
\[\del_{ij}P = \frac{P-s_{ij}P}{x_i-x_j}\]
for distinct $i$ and $j$. We abbreviate $\del_{i,i+1}$ by $\del_i$.

The following proposition describes some simple but important properties of the divided difference operators.
\begin{prop} \label{prop-relations}
	The divided difference operators satisfy, for distinct $i$, $j$, $k$ and $l$:
	\begin{enumerate}[(a)]
		\item $\del_{ij} = -\del_{ji}$;
		\item $\del_{ij}^2 = 0$;
		\item $\del_{ij}\del_{kl} = \del_{kl}\del_{ij}$;
		\item $\del_{ij}\del_{jk} + \del_{jk}\del_{ki} + \del_{ki}\del_{ij} = 0$;
		\item $\del_{ij}\del_{jk}\del_{ij} = \del_{jk}\del_{ij}\del_{jk} = \del_{ij}\del_{ik}\del_{jk} = \del_{jk}\del_{ik}\del_{ij}$;
		\item $\del_{ij}(PQ) = \del_{ij}P \cdot Q + s_{ij}P \cdot \del_{ij}Q$; and 
		\item $\del_{ij}w = w\del_{w^{-1}(i)w^{-1}(j)}$ for all $w \in \SS_n$.
	\end{enumerate}
\end{prop}
\begin{proof}
	Straightforward computation.
\end{proof}

Proposition~\ref{prop-relations}(e) follows from (a), (b), and (d). The first equality is the \emph{braid relation}, while the last equality is sometimes called the \emph{Yang-Baxter equation}.

The divided difference operators satisfy other relations that are not implied by those in Proposition~\ref{prop-relations}, but we will not need them here. However, the only relations between the simple divided difference operators $\del_i$ are (b), (c), and the braid relation in (e). These relations define the \emph{nil-Coxeter algebra} of $\SS_n$. Thus $\del_1, \dots, \del_{n-1}$ generate a faithful representation of the nil-Coxeter algebra (see \cite{FS}). Given a permutation $w \in \SS_n$, let $w = s_{i_1}\dotsm s_{i_\ell}$ be a reduced expression, and define $\del_w = \del_{i_1}\dotsm \del_{i_\ell}$. Then $\del_w$ does not depend on the choice of reduced expression for $w$. For any two permutations $v, w \in \SS_n$, $\del_v\del_w = \del_{vw}$ if $\ell(vw) = \ell(v)+\ell(w)$ and 0 otherwise.

By applying Proposition~\ref{prop-relations}(f) repeatedly to $\del_w(PQ) = \del_{i_1}\dotsm \del_{i_\ell}(PQ)$, we can express $\del_w(PQ)$ as
\[\del_w(PQ) = \sum_{v \in \SS_n} v(\del_{w/v} P) \cdot \del_v Q\]
for some linear operators $\del_{w/v}$. These operators are called the \emph{skew divided difference operators}. The operator $\del_{w/v}$ reduces the degree of a polynomial by $\ell(w)-\ell(v)$.

\begin{ex} \label{ex-skew}
	Let $w = s_1s_2 \in \SS_3$. Then
	\begin{align*}
		\del_w(PQ) &= \del_1\del_2(PQ)\\
		&= \del_1(\del_2P\cdot Q + s_2P\cdot\del_2Q)\\
		&= \del_1\del_2P\cdot Q + s_1\del_2P\cdot\del_1Q + \del_1s_2P\cdot\del_2Q+s_1s_2P\cdot\del_1\del_2Q\\
		&= \del_1\del_2P\cdot Q + s_1\del_2P\cdot\del_1Q + s_2\del_{13}P\cdot\del_2Q+s_1s_2P\cdot\del_1\del_2Q.
	\end{align*}
	Hence $\del_{w/id} = \del_1\del_2$, $\del_{w/s_1} = \del_2$, $\del_{w/s_2} = \del_{13}$, and $\del_{w/w} = 1$.
\end{ex}

Alternatively, $\del_{w/v}$ can be calculated as follows. Again consider any reduced expression $w = s_{i_1}\dots s_{i_\ell}$. For any subset $J \subset \{1, \dots, \ell\}$, let $\varphi_J = \prod_{j=1}^\ell \varphi_j(J)$, where $\varphi_j(J) = s_{i_j}$ if $j \in J$ and $\del_{i_j}$ if $j \not \in J$. Then
\begin{equation*}\label{eq-del}
\del_{w/v} = v^{-1}\sum_J \varphi_J,
\end{equation*}
where $J$ ranges over all subsets for which the product $\prod_{j \in J} s_{i_j}$ is a reduced expression for $v$. Note that this immediately implies that $\del_{w/v}=0$ unless $v<w$ in Bruhat order.

By using Proposition~\ref{prop-relations}(g) to collect the transpositions in $\varphi_J$ to the left, we will always be able to write $\del_{w/v}$ as a polynomial in the divided difference operators $\del_{ij}$.

\begin{ex} \label{ex-skew2}
	Let $w = s_2s_1s_3s_2 \in \SS_4$, and let $v=s_2$. Then the only possibilities for $J$ that yield a reduced expression for $v$ are $J=\{1\}$ and $J=\{4\}$. Thus
	\begin{align*}
	\del_{w/v} &= v^{-1}(s_2\del_{12}\del_{34}\del_{23}+\del_{23}\del_{12}\del_{34}s_2)\\
	&= s_2(s_2\del_{12}\del_{34}\del_{23} + s_2\del_{32}\del_{13}\del_{24})\\
	&= \del_{12}\del_{34}\del_{23} - \del_{23}\del_{13}\del_{24}. 
	\end{align*}
	Note that while this expression has a negative sign in it, we can use Proposition~\ref{prop-relations}(d) to rewrite it as
	\begin{align*}
	\del_{12}\del_{34}\del_{23}-\del_{23}\del_{13}\del_{24}
	&= \del_{12}(\del_{23}\del_{24} + \del_{24}\del_{34}) - (\del_{12}\del_{23} - \del_{13}\del_{12})\del_{24}\\
	&= \del_{12}\del_{24}\del_{34} + \del_{13}\del_{12}\del_{24}.
	\end{align*}
\end{ex}

This example shows that na\"ive evaluation of $\del_{w/v}$ will generally not give an expression in $\del_{ij}$ with $i<j$ that has positive coefficients. The main result of this paper will be to prove the following theorem, which states that such a positive expression always exists. It will follow as an immediate consequence of Theorem~\ref{thm-main} below.
\begin{thm} \label{thm-maindd}
	For any $v,w \in \SS_n$, the skew divided difference operator $\del_{w/v}$ can be written as a polynomial with nonnegative coefficients in the operators $\del_{ij}$ for $i<j$.
\end{thm}
An explicit expression for $\del_{w/v}$ follows from Corollary~\ref{cor-explicit}.

\subsection{Fomin-Kirillov algebra}
We will not need the full set of relations between the $\del_{ij}$ to prove Theorem~\ref{thm-maindd}, only the ones appearing in Proposition~\ref{prop-relations}. As such, it will be helpful to work inside the \emph{Fomin-Kirillov algebra} $\FK_n$, which is essentially defined by the relations in Proposition~\ref{prop-relations}(a)--(d). (Relation (e) can be derived from these four.) See \cite{BLM, FK} for more information.

\begin{defn}
	The Fomin-Kirillov algebra $\FK_n$ is the (noncommutative) algebra with generators $x_{ij} = -x_{ji}$ for $1 \leq i < j \leq n$ satisfying the following relations for distinct $i$, $j$, $k$, and $l$:
	\begin{itemize}
		\item $x_{ij}^2=0$;
		\item $x_{ij}x_{kl} = x_{kl}x_{ij}$; and
		\item $x_{ij}x_{jk} + x_{jk}x_{ki} + x_{ki}x_{ij} = 0$.
	\end{itemize}
\end{defn}

The divided difference operators therefore give a representation of $\FK_n$, though this representation is not faithful for $n \geq 3$. As with divided difference operators, the subalgebra generated by $x_{12}, x_{23}, \dots, x_{n-1,n}$ is isomorphic to the nil-Coxeter algebra. Given any $w \in \SS_n$ with reduced expression $w=s_{i_1}\dotsm s_{i_\ell}$, we will define $x_w = x_{i_1,i_1+1} \dotsm x_{i_\ell, i_\ell+1} \in \FK_n$. We can similarly define $x_{w/v} \in \FK_n$ as before (or one can take Proposition~\ref{prop-skew} below as a definition).

One important property of $\FK_n$ is that it has a large amount of structure: in particular, it is a \emph{braided Hopf algebra}. (This is not the case for the algebra generated by divided difference operators, which is a quotient of $\FK_n$.) We review some properties of this structure below. For more details about this structure and braided Hopf algebras in general, see \cite{AS, BLM, FP, MS}.

\subsubsection{Grading and braiding} In addition to the usual degree grading, the Fomin-Kirillov algebra has an $\SS_n$-grading: define the $\SS_n$-degree of $x_{ij}$ to be $s_{ij}$ and extend by multiplicativity. We will write $s_P$ for the $\SS_n$-degree of an $\SS_n$-homogeneous element $P \in \FK_n$. We will use the word ``homogeneous'' to mean homogeneous with respect to both the usual grading and the $\SS_n$-grading.

The Fomin-Kirillov algebra also has an $\SS_n$-action given by permuting the indices of the generators $x_{ij}$. In other words, for $w \in \SS_n$, $w(x_{ij}) = x_{w(i)w(j)}$, and we extend by multiplicativity. This induces an automorphism of $\FK_n$.

Using this gradation and action, $\FK_n$ can be thought of as an object in a braided monoidal category (or more specifically, in the \emph{Yetter-Drinfeld category} over $\CC[S_n]$). In other words, we define a braiding $\tau\colon \FK_n \otimes \FK_n \to \FK_n \otimes \FK_n$ by $\tau(P \otimes Q) = Q \otimes s_Q^{-1}(P)$ for homogeneous $P, Q \in \FK_n$. We use this to define a braided product structure on $\FK_n \otimes \FK_n$ via \[(P_1 \otimes P_2) (Q_1 \otimes Q_2) = P_1Q_1 \otimes s_{Q_1}^{-1}(P_2)Q_2.\]

\begin{rmk}
	This convention for the braiding is different from the usual convention, but we use it here for ease of compatibility with the definitions of skew divided difference operators.
\end{rmk}

\subsubsection{Coproduct} There exists a coproduct $\Delta$ on $\FK_n$ defined by $x_{ij} \mapsto x_{ij} \otimes 1 + 1 \otimes x_{ij}$ and extended to all of $\FK_n$ as a braided homomorphism.

\begin{ex}
	We can compute $\Delta(x_{12}x_{23})$:
	\begin{align*}
		\Delta(x_{12}x_{23}) &= (x_{12} \ot 1 + 1 \ot x_{12})(x_{23} \ot 1 + 1 \ot x_{23})\\
		&= (x_{12} \ot 1)(x_{23} \ot 1) + (x_{12} \ot 1)(1 \ot x_{23}) + (1 \ot x_{12})(x_{23} \ot 1) + (1 \ot x_{12})(1 \ot x_{23})\\
		&= x_{12}x_{23} \ot 1 +  x_{12} \ot x_{23} + x_{23} \ot x_{13} + 1 \ot x_{12}x_{23}.
	\end{align*}
	Compare this calculation to Example~\ref{ex-skew}.
\end{ex}

The main reason for introducing the coproduct structure is the following proposition.
\begin{prop} \label{prop-skew}
	For any $w \in \SS_n$, $\Delta(x_w) = \sum_{v\in \SS_n} x_v \ot x_{w/v}$.
\end{prop}
\begin{proof}
	Follows from the definition of $x_{w/v}$ and the coproduct structure.
\end{proof}

Note that if $P \in \FK_n$ is a monomial, then any term appearing in the first tensor factor of $\Delta(P)$ is the product of a subsequence of variables in $P$. (This is not the case for the second tensor factor due to the braiding.)

\subsubsection{Pairing} There exists a unique linear map $\Delta_{ab}\colon \FK_n \to \FK_n$ satisfying
\[\Delta_{ab}(x_{ij}) = \begin{cases} 1, & \text{if }i=a, j=b;\\
-1, & \text{if }i=b,j=a;\\
0, & \text{otherwise;}\end{cases}\]
and $\Delta_{ab}(PQ) = \Delta_{ab}(P)\cdot Q + s_{ab}(P) \cdot \Delta_{ab}(Q)$ for all $P,Q \in \FK_n$. The operators $\Delta_{ab}$ satisfy the relations of $\FK_n$, so they define a left action of $\FK_n$ on itself. We can think of $\Delta_{ab}$ as having degree $-1$ and $\SS_n$-degree $s_{ab}$. For any $P \in \FK_n$, we will write $\Delta_P$ for the corresponding operator; in other words, if $P = x_{i_1j_1}\dotsm x_{i_kj_k}$, then $\Delta_P = \Delta_{i_1j_1} \dotsm \Delta_{i_kj_k}$, and we extend by linearity. 

There is likewise a dual action: one can define $\nabla_{ab}\colon \FK_n \to \FK_n$ (acting on the right) satisfying $ (x_{ij})\nabla_{ab}=\Delta_{ab}(x_{ij})$ and $(PQ)\nabla_{ab} = P\cdot (Q)\nabla_{ab} + (P)(s_Q\nabla_{ab}) \cdot Q$, where $s_Q\nabla_{ab} = \nabla_{s_Q(a)s_Q(b)}$ for homogeneous $P,Q \in \FK_n$. Then the operators $\nabla_{ab}$ define a right action of $\FK_n$ on itself. We define $\nabla_P$ for any $P \in \FK_n$ as above.

Note that when $\nabla_{ab}$ acts on a monomial $P$, it results in a linear combination of monomials obtained from $P$ by removing a variable. In other words, $(P)\nabla_{ab}$, and in general $(P)\nabla_Q$, has an expression that only contains variables appearing in $P$. (This is not the case for the action of $\Delta_{ab}$ due to the twisting action.)

\begin{ex}
	Here are two computations involving these operators:
	\begin{align*}
		\Delta_{23}(x_{12}x_{23}x_{12}) &= \Delta_{23}(x_{12}) \cdot x_{23}x_{12} + x_{13} \cdot \Delta_{23}(x_{23}) \cdot x_{12} + x_{13}x_{32} \cdot \Delta_{23}(x_{12})\\& = x_{13}x_{12},\\
		(x_{12}x_{23}x_{12})\nabla_{23} &= x_{12}x_{23} \cdot (x_{12})\nabla_{23} + x_{12} \cdot (x_{23})\nabla_{13} \cdot x_{12} + (x_{12})\nabla_{12} \cdot x_{23}x_{12}\\
		&= x_{23}x_{12}.
	\end{align*}
\end{ex}

If $P$ and $Q$ are homogeneous of the same degree, then $\Delta_P(Q) = \Delta_Q(P) = (P)\nabla_Q = (Q)\nabla_P$. If we write $\langle P, Q \rangle = \Delta_P(Q)$, then this defines a symmetric bilinear form $\langle \cdot, \cdot \rangle$ on $\FK_n$. With respect to this form, $\Delta_P$ and $\nabla_P$ are adjoint to right and left multiplication by $P$, respectively. If $P$ and $Q$ are homogeneous, then $\langle P, Q \rangle = 0$ unless $\deg P = \deg Q$ and $s_P = s_Q^{-1}$.

An alternative way to describe $\Delta_P$ and $\nabla_P$ is in terms of this bilinear form and the coproduct. If $\Delta(Q) = \sum Q^i_{(1)} \ot Q^i_{(2)}$, then
\begin{align*}
\Delta_P(Q) &= \sum \langle P, Q^i_{(1)} \rangle \cdot Q^i_{(2)},\\
(Q)\nabla_P &= \sum Q^i_{(1)} \cdot \langle Q^i_{(2)}, P \rangle.
\end{align*}
Note that by the cocommutativity of the coproduct, $\Delta_{P_1}$ and $\nabla_{P_2}$ commute: if $(\Delta \ot \Delta)(Q) = \sum Q^i_{(1)} \ot Q^i_{(2)} \ot Q^i_{(3)}$, then
\[\Delta_{P_1}(Q)\nabla_{P_2} = \sum \langle P_1, Q^i_{(1)} \rangle \cdot Q^i_{(2)} \cdot \langle Q^i_{(3)}, P_2 \rangle.\]

\subsubsection{Antipode} The antipode $S\colon \FK_n \to \FK_n$ is defined by $x_{ij} \mapsto -x_{ij}$, extended to all of $\FK_n$ as a braided antihomomorphism. In other words, if $\mu\colon \FK_n \ot \FK_n \to \FK_n$ is the multiplication map, then
\[S(PQ) = S(\mu(P \ot Q)) = \mu(\tau(S(P) \ot S(Q))) = \mu(S(Q) \ot s_Q^{-1}(S(P))) = S(Q) \cdot s_Q^{-1}(S(P)).\]
The antipode preserves $\SS_n$-degree.

\begin{ex}
	The antipode of $x_{12}x_{23}x_{34}$ is
	\begin{align*}
		S(x_{12}x_{23}x_{34}) &= S(x_{34}) \cdot s_{34}(S(x_{12}x_{23}))\\
		&= S(x_{34}) \cdot s_{34}(S(x_{23})) \cdot s_{34}s_{23}(S(x_{12}))\\
		&= (-x_{34}) \cdot (-x_{24}) \cdot (-x_{14})\\
		&= -x_{34}x_{24}x_{14}.
	\end{align*}
\end{ex}

It will be helpful for us to introduce a variant of the antipode. Let $\rho\colon \FK_n \to \FK_n$ be the map that reverses the order of any monomial, and for any homogeneous $P$, let $\Sbar(P) = (-1)^{\deg P}\rho(S(P))$. For example, $\Sbar(x_{12}x_{23}x_{34}) = x_{14}x_{24}x_{34}$. Note that $s_{\Sbar(P)} = s_P^{-1}$.

From this definition, it is easy to check that $\Sbar(x_{i_1j_1} \cdots x_{i_\ell j_\ell}) = y_1\cdots y_\ell$, where $y_k = s_{i_\ell j_\ell} \cdots s_{i_{k+1}j_{k+1}}(x_{i_kj_k})$.

The following proposition gives some important properties of $\Sbar$. 
\begin{prop}\label{prop-sbar}
	\begin{enumerate}[(a)]
		\item For homogeneous $P,Q \in \FK_n$, $\Sbar(PQ) = s_Q^{-1}(\Sbar(P)) \cdot \Sbar(Q)$.
		\item The map $\Sbar$ is an involution.
		\item Let $\overline{\tau} \colon \FK_n \ot \FK_n \to \FK_n \ot \FK_n$ be the linear map that switches the two tensor factors (without twisting). Then $\Delta\circ\Sbar = \taubar \circ (\Sbar \ot \Sbar) \circ \Delta$.
		\item For any $P \in \FK_n$, $\Delta_{ab}(\Sbar(P)) = \Sbar((P)\nabla_{ab})$. 
		\item The operators $\Sbar$ and $\rho$ are adjoint with respect to $\langle \cdot, \cdot \rangle$.
	\end{enumerate}
\end{prop}
\begin{proof}
	For (a),
	\begin{align*}
	\Sbar(PQ) &= (-1)^{\deg PQ} \cdot \rho(S(PQ))\\
	& = (-1)^{\deg P + \deg Q}\cdot \rho(S(Q) \cdot s_Q^{-1}(S(P))\\
	& = (-1)^{\deg P}s_{Q}^{-1}(\rho(S(P)) \cdot (-1)^{\deg Q}\rho(S(Q)) \\
	& = s_Q^{-1}(\Sbar(P)) \cdot \Sbar(Q).
	\end{align*}
	
	For (b), we induct on degree. It is clear that $\Sbar^2$ is the identity in degree 0 and 1. Then for higher degrees,
	\[\Sbar^2(PQ) = \Sbar(s_Q^{-1}(\Sbar(P)) \cdot \Sbar(Q)) = s_{\Sbar(Q)}^{-1}\Sbar(s_Q^{-1}(\Sbar(P))) \cdot \Sbar^2(Q) = s_Qs_Q^{-1}(\Sbar^2(P)) \cdot Q = PQ.\]
	
	For (c), we again induct on degree. Again the claim is clear in degree 0 or 1. Suppose $\Delta(P) = \sum P^i_{(1)} \ot P^i_{(2)}$ and $\Delta(Q) = \sum Q^j_{(1)} \ot Q^j_{(2)}$. Then by induction,
	\begin{align*}
		\Delta \circ \Sbar(PQ) &= \Delta(s_Q^{-1}(\Sbar(P)) \cdot \Sbar(Q))\\
		&= s_Q^{-1}\Delta(\Sbar(P)) \cdot \Delta(\Sbar(Q))\\
		&= s_Q^{-1}(\taubar \circ (\Sbar \ot \Sbar) \circ \Delta(P)) \cdot \taubar \circ (\Sbar \ot \Sbar) \circ \Delta(Q)\\
		&= \sum_{i,j} s_Q^{-1} (\Sbar(P^i_{(2)}) \ot \Sbar(P^i_{(1)})) \cdot (\Sbar(Q^j_{(2)}) \ot \Sbar(Q^j_{(1)}))\\
		&= \sum_{i,j} s_Q^{-1}(\Sbar(P^i_{(2)}))\Sbar(Q^j_{(2)}) \ot s_{\Sbar(Q^j_{(2)})}^{-1}(s_Q^{-1}(\Sbar(P^i_{(1)})))\Sbar(Q^j_{(1)})\\
		&= \sum_{i,j} s_{Q^j_{(2)}}^{-1} \Sbar(s_{Q^j_{(1)}}^{-1}(P^i_{(2)}))\Sbar(Q^j_{(2)}) \ot s_{Q^j_{(1)}}^{-1}(\Sbar(P^i_{(1)}))\Sbar(Q^j_{(1)})\\
		&= \sum_{i,j} \Sbar(s_{Q^j_{(1)}}^{-1}(P^i_{(2)})Q^j_{(2)}) \ot \Sbar(P^i_{(1)}Q^j_{(1)})\\
		&= \sum_{i,j}\taubar \circ (\Sbar \ot \Sbar)(P^i_{(1)}Q^j_{(1)} \ot s^{-1}_{Q^j_{(1)}}P^i_{(2)}Q^j_{(2)})\\
		&= \taubar \circ (\Sbar \ot \Sbar) \circ \Delta(PQ).
	\end{align*}
	
	For (d),
	\begin{align*}
		\Delta \circ \Sbar(P) &= \taubar \circ (\Sbar \ot \Sbar) \circ \Delta(P)\\
		&= \sum_i \taubar \circ (\Sbar \ot \Sbar)(P^i_{(1)} \ot P^i_{(2)})\\
		&= \sum_i \Sbar(P^i_{(2)}) \ot \Sbar(P^i_{(1)}).
	\end{align*}
	Thus
	\[
		\Delta_{ab}(\Sbar(P)) = \sum_i \langle x_{ab}, \Sbar(P^i_{(2)}) \rangle \cdot \Sbar(P^i_{(1)})
		= \sum_i \langle x_{ab}, P^i_{(2)} \rangle \cdot \Sbar(P^i_{(1)})
		= \Sbar((P) \nabla_{ab}).
	\]
	
	For (e), if $P$ and $Q$ are homogeneous of the same degree, then
	\[\langle Q, \Sbar(P) \rangle = \Delta_Q(\Sbar (P)) = \Sbar((P) \Delta_{\rho(Q)}) = (P)\Delta_{\rho(Q)} = \langle P, \rho(Q) \rangle.\qedhere\]
\end{proof}
\begin{rmk}
	Propostion~\ref{prop-sbar}(e) easily implies that the antipode $S$ is self-adjoint, but we will not need this result below.
\end{rmk}
Since we did not use any of the relations of $\FK_n$ in proving Proposition~\ref{prop-sbar}, the analogous result holds even in the full tensor algebra with no relations. In particular, for part (c), if $P$ is a monomial of degree $d$, then the $2^{d}$ terms obtained from expanding $\Delta \circ \Sbar(P)$ are identical to the $2^{d}$ terms obtained from expanding $\taubar \circ (\Sbar \ot \Sbar) \circ \Delta(P)$ without using any relations between the $x_{ij}$ (even $x_{ij} = -x_{ji}$).

In the next section, we will use the properties presented above to prove Theorem~\ref{thm-maindd}.

\section{Positivity}
Let us denote by $\FK_n^+$ the positive span of monomials in variables $x_{ij}$ with $i<j$. In this section, we will prove that $x_{w/v}\in \FK_n^+$ as conjectured by Kirillov \cite{Kirillov}. This will immediately imply Theorem~\ref{thm-maindd}.

First we relate the operators and pairing described in the previous section to the Bruhat order.

\begin{prop} \label{prop-nabla}
	\begin{enumerate}[(a)]
		\item Let $w \in \SS_n$, and choose any $x_{ij} \in \FK_n^+$. Then $(x_w)\nabla_{ij} = x_{ws_{ij}}$ if $ws_{ij} \lessdot w$ and 0 otherwise.
		\item Let $v,w \in \SS_n$. Then $(x_w)\nabla_v = x_{v'}$ if there exists a reduced factorization $w = v' \cdot v^{-1}$ and 0 otherwise.
	\end{enumerate}
\end{prop}
\begin{proof}
	Let $w = s_{a_1}\cdots s_{a_\ell}$. From the definition of coproduct, the degree $(\ell-1, 1)$ part of $\Delta(x_w)$ is 
	\[
	\sum_{k=1}^\ell x_{a_1,a_1+1} \cdots \widehat x_{a_k,a_k+1} \cdots x_{a_\ell, a_\ell+1} \ot s_{a_\ell} \cdots s_{a_{k+1}}(x_{a_k,a_k+1}).
	\]
	The first tensor factor is nonzero if and only if $s_{a_1} \cdots \widehat s_{a_k} \cdots s_{a_\ell}$ is a reduced expression for some $v \in \SS_n$ with $v \lessdot w$, and in this case, the second tensor factor must be, up to sign, $x_{v^{-1}w}$.
	
	If $v \lessdot w$, then by the strong exchange condition for Bruhat order, there is a unique $k$ such that $s_{a_1} \cdots \widehat s_{a_k} \cdots s_{a_\ell}$ is a reduced expression for $v$. Hence $(x_w)\nabla_{ij} = 0$ unless $v=ws_{ij} \lessdot w$, in which case $(x_w)\nabla_{ij} = x_v \cdot \langle s_{a_\ell}\cdots s_{a_{k+1}}  (x_{a_k,a_k+1}), x_{ij} \rangle = \pm x_v$. But in fact this must have a positive sign: since $s_{a_k} \cdot s_{a_{k+1}} \cdots s_{a_\ell} > s_{a_{k+1}} \cdots s_{a_\ell}$, it follows that $s_{a_{\ell}} \cdots s_{a_{k+1}}(a_k) < s_{a_{\ell}} \cdots s_{a_{k+1}}(a_k+1)$, so the former must be $i$ and the latter $j$ (since $i<j$ by assumption).
	
	For part (b), choose any reduced expression for $v$ and apply part (a) repeatedly.
\end{proof}

By iterating this result, we can prove the following consequence.

\begin{prop}\label{prop-pair1}
	Let $w \in \SS_n$, and let $P=x_{i_1j_1}\dots x_{i_\ell j_\ell} \in \FK_n^+$ be any monomial. Let $v_k = s_{i_kj_k}s_{i_{k+1}j_{k+1}}\dotsm s_{i_\ell j_\ell}$.
	Then $\langle x_w, P \rangle = 1$ if \[\id\lessdot v_\ell\lessdot v_{\ell-1}\lessdot \dotsb \lessdot v_1 = w^{-1}\]
	is a saturated chain in the Bruhat order of $\SS_n$; otherwise $\langle x_w, P \rangle = 0$.
	
	In particular, for $v,w \in \SS_n$, $\langle x_v, x_w \rangle = 1$ if $w=v^{-1}$ and 0 otherwise.
\end{prop}
\begin{proof}
	Clearly for $\langle x_w, P \rangle$ to be nonzero, we must have $\ell = \ell(w) = \deg(P)$. We proceed by induction on $\ell$.
	
	Note $\langle x_w, P \rangle = \langle (x_w) \nabla_{i_1j_1}, P' \rangle$, where $P = x_{i_1j_1}P'$. This is nonzero if and only if, by Proposition~\ref{prop-nabla}, $ws_{i_1j_1} \lessdot w$ and, by induction, $\id \lessdot v_\ell\lessdot v_{\ell-1}\lessdot \dots \lessdot v_2 = (ws_{i_1j_1})^{-1}$ is a saturated Bruhat chain, in which case $\langle x_w, P \rangle = 1$. But then \[v_2 = s_{i_1j_1}w^{-1} \lessdot w^{-1} = s_{i_1j_1}v_2 = v_1,\]
	completing the Bruhat chain, as desired. 
\end{proof}

We will also need the following result about $\Sbar(x_w)$.
\begin{prop}\label{prop-sbarpos}
	Let $w \in \SS_n$. Then $\Sbar(x_w) \in \FK_n^+$. In fact, the variables appearing in $\Sbar(x_w)$ are precisely those $x_{ij}$ for which $i<j$ and $w(i)>w(j)$.
\end{prop}
\begin{proof}
	Let $w=s_{i_1} \dots s_{i_\ell}$ be a reduced expression, and let $w_k = s_{i_k} \cdots s_{i_\ell}$. Then by definition, $\Sbar(x_w) = y_1 \cdots y_\ell$, where $y_k = w_{k+1}^{-1}(x_{i_k,i_k+1})$. Since $w_{k+1}^{-1}s_{i_k} = w_k^{-1} \gtrdot w_{k+1}^{-1}$, we must have $a=w_{k+1}^{-1}(i_k)<w_{k+1}^{-1}(i_{k}+1)=b$, so $y_k=x_{ab} \in \FK_n^+$. Then $w_{k}(a) = i_k+1 > i_k = w_k(b)$. Since multiplying on the left by simple transpositions cannot remove inversions without decreasing length, it follows that $w(a) > w(b)$, as desired.
\end{proof}

Proposition~\ref{prop-sbarpos} implies that, in particular, $\Sbar(x_{w_0}) \in \FK_n^+$. In fact, we can say more about $\Sbar(x_{w_0})$. This will be related to the following definition from the theory of Coxeter groups (see \cite{BB, BFP}).
\begin{defn}
	A \emph{reflection ordering} for the transpositions $s_{ij} \in \SS_n$ is a total order $\prec$ such that for any $i<j<k$, $s_{ik}$ lies somewhere between $s_{ij}$ and $s_{jk}$.
\end{defn}
There is an equivalent formulation: $t_1\prec\dots\prec t_N$ is a reflection ordering if and only if there exists a reduced expression $w_0 = s_{i_1} \dotsm s_{i_N}$ such that $t_k = s_{i_N} \dotsm s_{i_{k+1}}s_{i_k}s_{i_{k+1}}\dots s_{i_N}$. In other words, if $s_{i_1j_1} \prec \dots \prec s_{i_Nj_N}$ is a reflection ordering, then $x_{i_1j_1} \cdots x_{i_Nj_N} = \Sbar(x_{w_0})$. We are now ready to prove the following proposition.

\begin{prop} \label{prop-w0}
	In $\FK_n$, $x_{w_0} = \Sbar(x_{w_0}) = x_{i_1j_1}\dotsm x_{i_Nj_N} \in \FK_n^+$, where $s_{i_1j_1}\prec \dots \prec s_{i_Nj_N}$ is any reflection ordering.
\end{prop}
\begin{proof}
	It suffices to prove the claim for a fixed reflection ordering, say
	\[s_{12}\prec s_{13}\prec s_{23}\prec s_{14}\prec s_{24}\prec s_{34}\prec  \cdots \prec s_{1n} \prec \cdots \prec s_{n-1,n}.\]
	
	We induct on $n$. Let $w_0'$ be the longest element of $S_{n-1}$, so $w = w_0' \cdot s_{n-1}s_{n-2} \cdots s_1$ is a reduced factorization. By the inductive hypothesis, $x_{w_0'} = x_{12}x_{13}x_{23} \cdots x_{n-2,n-1} \in \FK_{n-1}$. Since $w_0'$ has maximum length in $S_{n-1}$, $x_{w_0'}x_{i,i+1} = 0$ for $i=1, \dots, n-2$. Then
	\begin{align*}
	x_{w_0'} \cdot x_{1n}x_{2n}x_{3n} \cdots x_{n-1,n} &= x_{w_0'} \cdot (-x_{12}x_{1n} + x_{2n}x_{12}) \cdot x_{3n}\cdots x_{n-1,n}\\
	&= x_{w_0'}\cdot x_{2n}x_{12} \cdot x_{3n}\cdots x_{n-1,n}\\
	&= x_{w_0'} \cdot x_{2n}x_{3n}\cdots x_{n-1,n} \cdot x_{12}\\
	&= x_{w_0'} \cdot (-x_{23}x_{2n}+x_{3n}x_{23})\cdots x_{n-1,n} \cdot x_{12}\\
	&= x_{w_0'} \cdot x_{3n}\cdots x_{n-1,n} \cdot x_{23}x_{12}\\
	&= \cdots\\
	&= x_{w_0'} \cdot x_{n-1,n}x_{n-2,n-1} \cdots x_{23}x_{12}\\
	&= x_{w_0} \qedhere.
	\end{align*}
\end{proof}

We are now ready to prove our main theorem.

\begin{thm} \label{thm-main}
	For any $v,w \in \SS_n$, $x_{w/v} \in \FK_n^+$. 
\end{thm}
\begin{proof}
	Since $\ell(w_0) = \ell(w) + \ell((w_0w)^{-1})$, by Proposition~\ref{prop-nabla}(b), $(x_{w_0})\nabla_{w_0w} = x_w$. Then by Propositions~\ref{prop-skew} and \ref{prop-pair1}, $\Delta_{v^{-1}}(x_w) = \sum_{v'} \langle x_{v^{-1}}, x_{v'} \rangle \cdot x_{w/v'} = x_{w/v}$. Thus
	\[x_{w/v} = \Delta_{v^{-1}}(x_{w_0})\nabla_{w_0w}.\]
	Now by Propositions~\ref{prop-w0}, \ref{prop-sbar}(d), and \ref{prop-nabla}(b), \[\Delta_{v^{-1}}(x_{w_0}) = \Delta_{v^{-1}}(\Sbar(x_{w_0})) = \Sbar((x_{w_0})\nabla_v) = \Sbar(x_{w_0v}).\] 
	Write $\Delta(\Sbar(x_{w_0v})) = \sum X^i_{(1)} \ot X^i_{(2)}$. Then since $\Sbar(x_{w_0v}) \in \FK_n^+$ by Proposition~\ref{prop-sbarpos}, we also have $X^i_{(1)} \in \FK_n^+$. Moreover, by Propositions~\ref{prop-sbar}(c) and \ref{prop-skew},
	\[\sum X^i_{(2)} \ot X^i_{(1)} = \taubar \circ \Delta \circ \Sbar(x_{w_0v}) = (\Sbar \ot \Sbar) \circ \Delta(x_{w_0v}) = \sum_u \Sbar(x_u) \ot \Sbar(x_{w_0v/u}).\]
	Hence if $X^i_{(2)}$ is nonzero, then it equals $\Sbar(x_u)$ for some $u \in \SS_n$ and therefore lies in $\FK_n^+$. Then
	\[x_{w/v} = (\Sbar(x_{w_0v}))\nabla_{w_0w} = \sum X^i_{(1)} \cdot \langle X^i_{(2)}, x_{w_0w} \rangle.\]
	Since $\langle X^i_{(2)}, x_{w_0w} \rangle$ is either 0 or 1 by Proposition~\ref{prop-pair1}, $x_{w/v}$ lies in $\FK_n^+$, as desired.
\end{proof}

Tracing through the proof of Theorem~\ref{thm-main}, we can write down an explicit expression for $x_{w/v}$.
\begin{cor} \label{cor-explicit}
	Let $v,w \in \SS_n$. Choose any reduced expression $w_0v = s_{i_1} \cdots s_{i_\ell}$, and let $w_k = s_{i_k} \cdots s_{i_\ell}$. Define $y_k = w_{k+1}^{-1}(x_{i_k,i_k+1}) \in \FK_n^+$. Then 
	\[x_{w/v} = \sum_J \prod_{k \notin J} y_k,\]
	where $J \subset \{1, \dots, \ell\}$ ranges over all subsets such that $\prod_{k \in J} s_{i_k}$ is a reduced expression for $w_0w$.
\end{cor}
\begin{proof}
	As in Theorem~\ref{thm-main}, $x_{w/v} = (\Sbar(x_{w_0v}))\nabla_{w_0w} = (y_1 \cdots y_\ell) \nabla_{w_0w}$. By the proof of Theorem~\ref{thm-main}, 
	\[\Delta(y_1 \cdots y_\ell) = \sum_{J \subset \{1, \dots, \ell\}}\left(\left(\textstyle\prod_{k \notin J} y_k\right) \ot \Sbar\left(\textstyle\prod_{k \in J} x_{i_k, i_k+1}\right)\right).\]
	By Proposition~\ref{prop-sbar}(e) $\langle \Sbar(x_u), x_{w_0w} \rangle = \langle x_u, x_{w^{-1}w_0} \rangle$, which by Proposition~\ref{prop-pair1} equals 1 if $u = w_0w$ and 0 otherwise. The result follows.
\end{proof}

\begin{ex}
	Let $w = s_2s_1s_3s_2 \in \SS_4$ and $v = s_2$. One reduced expression for $w_0v$ is $s_3s_2s_1s_2s_3$. Then $\Sbar(x_{w_0v}) = x_{13}x_{12}x_{14}x_{24}x_{34}$. Since $w_0w = s_1s_3 = s_3s_1$, there are two reduced subexpressions for $w_0w$ in $w_0v$, namely for $J=\{1,3\}$ and $J=\{3,5\}$. Removing the corresponding variables from $\Sbar(x_{w_0v})$ gives 
	\[x_{w/v} = x_{12}x_{24}x_{34} + x_{13}x_{12}x_{24}.\]
	Compare this calculation to Example~\ref{ex-skew2}.
\end{ex}

Each term in the expansion of $x_{w/v}$ corresponds to a reduced subword for $w_0w$ lying inside a reduced expression for $w_0v$. For more information about these subwords, see \cite{KM}.

Interestingly, Corollary~\ref{cor-explicit} implies that $x_{w/v}$ has an expression in which no variable is repeated in any monomial.

One can also use Theorem~\ref{thm-main} to give a positive recurrence for $x_{w/v}$.
\begin{cor} \label{cor-recurrence}
	Let $v,w \in \SS_n$, and suppose $a=v^{-1}(i)<v^{-1}(i+1)=b$. Let $v' = s_i v$ and $w' = s_iw$. If $w \lessdot w'$, then $x_{w/v} = x_{ab}x_{w/v'} + x_{w'/v'}$; otherwise $x_{w/v} = x_{ab}x_{w/v'}$.
\end{cor}
\begin{proof}
	If $v^{-1}(i)<v^{-1}(i+1)$, then $(w_0v)^{-1}(n+1-i)<(w_0v^{-1})(n-i)$. Hence there exists a reduced expression for $w_0v$ that starts with $s_{n-i}$. Let us fix such a reduced expression and apply Corollary~\ref{cor-explicit}. In this case $y_1 = x_{ab}$, and $y_2 \cdots y_\ell = \Sbar(x_{s_{n-i}w_0v}) = \Sbar(x_{w_0v'})$. 
	
	Suppose $J \subset \{1, \dots, \ell\}$ gives a reduced subexpression for $w_0w$. If $1 \notin J$, then $J$ gives a reduced subexpression for $w_0w$ in a reduced expression for $w_0v'$. This then contributes $x_{ab}x_{w/v'}$ to $x_{w/v}$. If $1 \in J$, then we must have $w_0w' = s_{n-i}w_0w \lessdot w_0w$, or equivalently $w \lessdot w'$, and $J \backslash \{1\}$ must give a reduced subexpression for $w_0w'$ in a reduced expression for $w_0v'$. Hence this contributes $x_{w'/v'}$ to $x_{w/v}$ (provided $w \lessdot w'$). The result follows.
\end{proof}

Corollary~\ref{cor-recurrence} then gives an explicit positive recurrence that implies $x_{w/v} \in \FK_n^+$ (noting that if $\ell(v)=\ell(w)$, then $x_{w/v} = 1$ if $v=w$ and 0 otherwise).

\bibliography{skewdd}
\bibliographystyle{plain}
\end{document}